\documentclass[11pt]{article}
        \usepackage[utf8]{inputenc}
        \usepackage{amsmath}
        \usepackage{amsfonts}
        \usepackage{amssymb}
        \usepackage{amsthm}
        \usepackage{mathrsfs}
        \usepackage{lscape}
        \usepackage{euscript}
        \usepackage{latexsym,enumerate,color,mdwlist}
        \usepackage{tikz}
        \usepackage{a4wide}
        \usepackage{hyperref}
        \usepackage{soul}
        \usepackage{authblk}
        \usepackage{mathtools}
        
        \usepackage[shortlabels]{enumitem}
        \setlist[enumerate]{leftmargin=25pt}
        \setlist[itemize]{leftmargin=25pt}
        
        \allowdisplaybreaks
        
        \theoremstyle{plain} \newtheorem{Thm}{Theorem}[section]
        \newtheorem{prop}[Thm]{Proposition}
        \newtheorem{lem}[Thm]{Lemma}

        \theoremstyle{definition}

        \newtheorem*{theorem*}{Theorem}
        \theoremstyle{remark}

        \newcommand{\F}{{\mathbb F}}

        \newcommand{\comment}[1]{}

        \def\Z{\mathbb{Z}}

\begin{document}
        \title{A new proof of R\'edei's theorem on the number of directions.}
        \date{}
\author{G\'abor Somlai \thanks{The author is a J\'anos Bolyai research fellow. The work of the author on the project leading to this application has received
funding from the European Research Council (ERC) under the European Union’s Horizon 2020
research and innovation programme (grant agreement No. 741420) \\and by  (NKFIH) Grant No. K138596 and  SNN 132625.}}
        \affil{E\"otv\"os Lor\'and University, Department of Algebra and Number Theory \\
        gabor.somlai@ttk.elte.hu}
        \maketitle
        \begin{abstract}
        Rédei and Megyesi proved that the number of directions determined by a $p$ element subset of $\F_p^2$ is either $1$ or at least $\frac{p+3}{2}$.  The same result was independently obtained by Dress, Klin and Muzychuk. We give a new and short proof of this result using a Lemma proved by Kiss and the author. The new proof further on a result on polynomials over finite fields. 
        \end{abstract}
        
        \section{Introduction}
        Let $p$ be a prime.
        The points of the projective line $PG(1,\F_p)$ can be considered as equivalence classes of the non-zero vectors of the affine plane $AG(2,\F_p)$, where $\F_p$ denotes the field of $p$ elements.  Two elements are equivalent if one of them is a non-zero multiple of the other. For a subset $H$, the \emph{set of directions} $D(H) \subset PG(1,\F_p)$ is the equivalence classes corresponding  to the elements of $(H-H)\setminus \{ 0\}$. The number of directions determined by is $H$ is the cardinality of $D(H)$.
        
        Rédei investigated the number of directions determined by a $p$-element subset of the 2 dimensional space $\F_p^2$ over the finite field $\F_p$ of $p$ elements. Using his results on lacunary polynomials, Rédei proved that such a subset is a line or determine at least $\frac{p+1}{2}$ directions. Later, Megyesi excluded the case of $\frac{p+1}{2}$ directions. These results can be summarized as follows, see \cite{redei}. 
        \begin{Thm}\label{thm:redei}
            If a set of $p$ points in $\F_p^2$ is not a line, then
it determines at least $\frac{p+3}{2}$ directions. 
        \end{Thm}
Sets having exactly $\frac{p+3}{2}$ were described by Lovász and Schrijver \cite{lovaszschrijver}. As a generalization of Rédei's result, Szőnyi \cite{szonyi} proved that if $k \le p$, then a $k$-element subset not lying in a line determines at least $\frac{k+3}{2}$ directions. Gács \cite{gacs} showed that there is another gap in the possible number of directions between $\frac{p+3}{2}$ and $\lfloor 2 \frac{p-1}{3}\rfloor+1$.

Note that the result of Megyesi and Rédei was independently obtained by Dress, Klin and Muzychuk \cite{dressklinmuzychuk}. They used this result to give a new proof of Burnside's theorem on permutation groups of prime degree. Another application of the results on the number of directions in group theory is due to Dona \cite{dona}, who used his result to add to the theory of growth in groups. Further, the connection of the set of directions in the affine plane to blocking sets on finite projective plans are also discussed in \cite{gacs}.

One of the main purposes of this paper is to give a new proof of Theorem \ref{thm:redei}. The other one is to prove Theorem \ref{thm:main}, which will immediately imply Theorem \ref{thm:redei}. 

Let $g \colon \F_p \to \F_p$ be a polynomial. Considering the polynomial function corresponding to $g$ we may assume that the degree of 
    $g$ is at most $p-1$. Further the elements of $\F_p$ can be considered as elements in $\{0,1, \ldots, p-1 \} \subset \Z$. Thus we may consider the sum of the values in $\mathbb{Z}$. If it is small enough, then we obtain restrictions on the degree of $g$.

    In order to motivate the following theorem, it is useful to consider the polynomial $q(x)=x^{\frac{p-1}{2}}+1$. The sum of the values of  $q$ is equal to $p$ since
\begin{equation}
    q(x)=
    \begin{cases}
       2 & \mbox{if } $x$ \mbox{ is a quadratic residue}\\
       1 & \mbox{if } $x=0$ \\ 
       0 &\mbox{otherwise.}
    \end{cases}
\end{equation}
This simple example shows that the following theorem is sharp. 
    \begin{Thm}\label{thm:main}
    Let $p$ be an odd prime. If $\sum_{i=1}^{p-1} g(i) =p $, then the degree of either $g$ is at least $\frac{p-1}{2}$ or $g$ is a constant function.
    \end{Thm}

\section{Technique}
The proof of Theorem \ref{thm:main} relies on the following result
proved in \cite{elsoiranyoscikk}. Note that the proof of this lemma uses Rédei's polnyomials. 
\begin{lem}\label{lem:kissomlai}
        Let $A$ be a subset of $\F_p^2$ of cardinality $kp$. Assume that it determines $d$ special directions. Let $r$ be a \emph{projection function} defined as follows: 
        $$r(i)=| \{ j \in \F_p \mid (i,j) \in A \} |.$$ Then $d \ge deg(r)+2 $.
        \end{lem}
Using suitable affine transformation we may prove the previous lemma for any projection function obtained in this way instead of the vertical projection.

    As a corollary of this lemma we obtain that in order to prove Rédei's result it is sufficient to prove the Theorem \ref{thm:main} that is of independent interest.

    Using a simple argument we get a weaker result than Theorem \ref{thm:main}. 
\begin{prop}\label{prop}
        Let $p$ be an odd prime. If $\sum_{i=1}^{p-1} g(i) =p $, then the degree of either $g$ is at least $\frac{p-1}{3}$ or $g$ is a constant function.
\end{prop}
\begin{proof}
We will simply prove that one of its values of $g$ is taken at least $\frac{p-1}{3}$ times. More precisely $|\{ x \in \F_p \mid g(x)=0\}|\ge \frac{p-1}{3}$ or $|\{ x \in \F_p \mid g(x)=1\}|\ge \frac{p-1}{3}$.   

Assume indirectly this is not the case. 
Then 
\begin{equation*}
        \sum_{x \in F_p} g(x) \ge 
    \sum_{x \in F_p \colon g(x) \ge 1} 1 + \sum_{x \in F_p \colon g(x) \ge 2} 1  \ge (p-\frac{p-1}{3}) + (p- 2 \frac{p-1}{3})=p+1,
\end{equation*}
a contradiction.
\end{proof}
In order to emphasize the usefulness of Lemma \ref{lem:kissomlai} we prove the following simple result.  The proof uses again the observation that the multiplicity of any element in the range of a non-constant polynomial is a lower bound for the degree of a non-zero polynomial.  
\begin{Thm}
    Let $H$ be a subset of $\F_p^2$ of cardinality $p$. Let $a$ and $b$ be the size of the projection of $H$ to the $x$ and $y$ axis, respectively. Then the number of directions determined by $H$ is at least $p-\min \{ a,b\}+2$. 
\end{Thm}
\begin{proof}
Let $r$ be the function from $\F_p$ to $\F_p$ be defined as in Lemma \ref{lem:kissomlai}. Then the multiplicity of $0$ as a root of $r$ is at least $p-a$ and for the projection polynomial defined by the projection of $H$ to the second coordinate gives a polynomial of degree at least $p-b$. Now, Lemma \ref{lem:kissomlai} gives  the result.   
\end{proof}
The importance of this trivial corollary of Lemma \ref{lem:kissomlai} relies on the similarity of this result to the one of Di Benedetto, Solymosi and White \cite{cartesianproduct1}, who proved that the number of directions determined by a subset of $\F_p^2$, which is the Cartesian product of the subsets $A, B \subset \F_p$ is at least 
$$ |A| \cdot |B|- \min\{|A|,|B| \} +2.
$$

\section{Proof of the main result}

    \begin{proof}
        As we have mentioned, every polynomial function from $\F_p$ to $\F_p$ coincides with a unique polynomial of degree at most $p-1$ so we will automatically reduce the degree below $p$.
        
       The proof of Theorem \ref{thm:main} relies on the following simple observation. 
The degree of a polynomial $h$ is smaller than $p-1$ if and only if 
    \[
  \sum_{y \in \F_p} h(y) \equiv 0 \pmod{p}. 
    \]    

Let us consider the polynomial $f(x)=g(x^2)$ (reduced to degree at most $p-1$).
Clearly, if  $\sum_{y \in \F_p} f(y) \not\equiv 0 \pmod{p}$, then $deg(g) \ge \frac{p-1}{2}$.

We argue that if there were a non-constant polynomial of degree less than $\frac{p-1}{2}$ such that the sum of its values is $p$, then there is one which takes 0 at 0. 
It is clear that if the polynomial is non-constant, then the sum can only be p if $0$ is in the range of the polynomial. Now applying a linear substitution $x \to x+i$ on the $x$ variable of the polynomial we obtain a polynomial of the same degree satisfying $f(0)=0$.

Let us first estimate the sum of the values of $f$ from above. 
\begin{equation}\label{eq1}
\begin{split}
    \sum_{y \in \F_p} f(y)&=\sum_{x \in \F_p} g(x^2)=g(0) + 2 \sum_{x \in (\F_p^*)^2} g(x)  \\
    &=  2 \sum_{x \in (\F_p^*)^2} g(x) \le 2  \sum_{x \in \F_p} g(x)=2p.
    \end{split}  
\end{equation}
It is clear that $\sum_{y \in \F_p} f(y)$ cannot be equal to $p$ since it is an even number by equation \eqref{eq1}. 

On the other hand, equality in \eqref{eq1} can only hold if $g$ vanishes on every the non-quadratic residues, when the degree of $g$ is at least $\frac{p-1}{2}+1=\frac{p+1}{2}$. 

It could be that the previous sum is 0 but then $g$ vanishes on the quadratic residues, again having many roots. 

Therefore we obtain that the sum in equation \eqref{eq1} is not divisible by $p$, finishing the proof of Theorem \ref{thm:main}. 
\end{proof}
Theorem \ref{thm:redei} can be applied for the projection function defined in Lemma \ref{lem:kissomlai}. We may assume that the set is not a vertical line. 
It follows from Theorem \ref{thm:main} that the degree of $r$ is at least $\frac{p-1}{2}$. Then by Lemma \ref{lem:kissomlai}, the number of direction determined by $A$ is at least $\frac{p+3}{2}=\frac{p-1}{2}+2$. 

There are natural problems arising here. Can we find a similar result proving the ones listed in the beginning of this paper. 
\begin{itemize}
    \item 
    It is true that up to affine transformations $x^{\frac{p-1}{2}}+1$ is the unique polynomial of degree $\frac{p-1}{2}$ such that the sum of its values is $p$? 
\item Is it possible to prove Gács's result on the number of directions?
\end{itemize}
\section*{Acknowledgement}
The author is grateful for Gergely Kiss and Zolt\'an Nagy for short but fruitful conversations.

        \end{document}